\pdfoutput=1
\RequirePackage{ifpdf}
\ifpdf
\documentclass[pdftex]{sigma}
\else
\documentclass{sigma}
\fi

\numberwithin{equation}{section}

\newtheorem{Theorem}{Theorem}[section]
\newtheorem{Corollary}[Theorem]{Corollary}
\newtheorem{Lemma}[Theorem]{Lemma}
\newtheorem{Proposition}[Theorem]{Proposition}

\begin{document}


\newcommand{\arXivNumber}{1508.06587}

\renewcommand{\PaperNumber}{007}

\FirstPageHeading

\ShortArticleName{Automorphisms of ${\mathbb C}^*$ Moduli Spaces Associated to a Riemann Surface}

\ArticleName{Automorphisms of $\boldsymbol{{\mathbb C}^*}$ Moduli Spaces\\ Associated to a Riemann Surface}

\Author{David {BARAGLIA}~$^\dag$, Indranil {BISWAS}~$^\ddag$ and Laura P.~{SCHAPOSNIK}~$^{\S}$}

\AuthorNameForHeading{D.~Baraglia, I.~Biswas and L.P.~Schaposnik}

\Address{$^\dag$~School of Mathematical Sciences, The University of Adelaide, Adelaide SA 5005, Australia}
\EmailD{\href{mailto:david.baraglia@adelaide.edu.au}{david.baraglia@adelaide.edu.au}}

\Address{$^\ddag$~School of Mathematics, Tata Institute of Fundamental Research,\\
\hphantom{$^\ddag$}~Homi Bhabha Road, Mumbai 400005, India}
\EmailD{\href{mailto:indranil@math.tifr.res.in}{indranil@math.tifr.res.in}}

\Address{$^{\S}$~Department of Mathematics, University of Illinois, Chicago, IL 60607, USA}
\EmailD{\href{mailto:schapos@uic.edu}{schapos@uic.edu}}

\ArticleDates{Received August 27, 2015, in f\/inal form January 15, 2016; Published online January 20, 2016}

\Abstract{We compute the automorphism groups of the Dolbeault, de Rham and Betti moduli spaces for
the multiplicative group ${\mathbb C}^*$ associated to a compact connected Riemann surface.}

\Keywords{holomorphic connection; Higgs bundle; character variety; automorphism}

\Classification{14H60; 14J50; 53D30}

\section{Introduction}\label{sec1}

Let $X$ be a compact connected Riemann surface of genus $g$, where $g \geq  1$.
Let ${\mathbb C}^* =  {\mathbb C}{\setminus}\{0\}$ be the multiplicative group. We are
interested in studying the automorphism groups of certain ${\mathbb C}^*$-moduli spaces
associated to $X$, arising from non-abelian Hodge theory. Namely these are the de Rham,
Betti and Dolbeault moduli spaces ${\mathcal M}_C$, ${\mathcal M}_R$, ${\mathcal M}_H$
parametrizing holomorphic ${\mathbb C}^*$-connections, representations of the fundamental
group into ${\mathbb C}^*$ and degree zero Higgs line bundles respectively. While these
three moduli spaces are all homeomorphic, their algebraic structures are quite dif\/ferent
(${\mathcal M}_C$ and ${\mathcal M}_H$ are not even biholomorphic)
and we f\/ind that their automorphism groups are also quite dif\/ferent.

In \cite{Ba}, a classif\/ication was obtained of the analytic automorphism groups of the
moduli space of ${\rm SL}(n,\mathbb{C})$-Higgs bundles, i.e., the ${\rm SL}(n,\mathbb{C})$
Dolbeault moduli space. It remains an open question to determine which of the analytic
automorphisms found in~\cite{Ba} are algebraic and also to determine the corresponding
automorphism groups for the ${\rm SL}(n,\mathbb{C})$ de Rham and Betti moduli spaces (note
that de Rham and Betti moduli spaces are analytically but not algebraically
isomorphic). As mentioned above, the goal of this paper is to address this classif\/ication problem for the
corresponding ${\mathbb C}^*$-moduli spaces. We leave the task of extending our
results to noncommutative reductive groups as an interesting and challenging open problem.

Motivation for studying the automorphisms of these moduli spaces arises from mirror
symmetry, the geometric Langlands program and their relation to physics, as promoted in
the celebrated work of Kapustin and Witten~\cite{KW}. Namely, one is interested in the
construction of examples of naturally def\/ined subvarieties of these moduli spaces,
known as {\em branes} in the language of physics. One way of constructing such
subvarieties which has proved fruitful is as the f\/ixed point set of an automorphism of
the moduli space, as seen in~\cite{BS1,BS2}. This has lead us to consider the problem
of determining the automorphism groups of these moduli spaces in order to see how
general our constructions are.

In what follows we shall describe the structure and results of this paper. We begin this
paper by studying in Section~\ref{sec2} the structure of the de Rham
moduli space ${\mathcal M}_C$ of holomorphic ${\mathbb C}^*$-connections on~$X$ up to
gauge equivalence, i.e., pairs~$(L,D)$ where $L$ is a holomorphic line bundle and~$D$ is
a holomorphic connection on~$L$. After recalling properties
of the space, we give in Proposition~\ref{prop1} a gauge theoretic
proof of the known result that every algebraic function on~$\mathcal{M}_{C}$ is constant.

The moduli space ${\mathcal M}_C$ is a complex algebraic group with multiplication given by taking the
tensor product of line bundles with connections, and thus ${\mathcal M}_C$ acts on itself
by translations giving an injective homomorphism
\begin{gather*}
\rho\colon \  {\mathcal M}_C \longrightarrow  \operatorname{Aut}({\mathcal M}_C)  ,
\end{gather*}
where $\operatorname{Aut}({\mathcal M}_C)$ denote the group of algebraic
automorphisms of ${\mathcal M}_C$. This map is considered in Section~\ref{sec3}, where we show the following (see Theorem~\ref{thm1}):

\begin{Theorem}
The quotient $\operatorname{Aut}({\mathcal M}_C)/(\rho({\mathcal M}_C))$ is a countable group.
In particular, the image of $\rho$ is the connected component of
$\operatorname{Aut}({\mathcal M}_C)$ containing the identity element.
\end{Theorem}

Let $J(X)$ be the Jacobian of $X$ and let $\rho_0 \colon J(X) \longrightarrow  \operatorname{Aut}(J(X))$ be the homomorphism given by letting $J(X)$ act on itself by translation.
In Section~\ref{sec3}, it is found that the quotient $\operatorname{Aut}({\mathcal
M}_C)/(\rho({\mathcal M}_C))$ can be identif\/ied with a subgroup of $\operatorname{Aut}(J(X))/\rho_0(J(X))$.

From non-abelian Hodge theory it is seen that the moduli space ${\mathcal M}_C$
carries a naturally def\/ined algebraic symplectic form~\cite{atiyah,Go}. Let $\theta \in H^2({\mathcal
M}_C , \mathbb{C})$ denote the cohomology class of the symplectic form and let
$\operatorname{Aut}_\theta({\mathcal M}_C)$ be the subgroup of $\operatorname{Aut}({\mathcal M}_C)$ preserving
$\theta$. In Section~\ref{sec3} we study this subgroup, and give its complete characterization in Theorem~\ref{thm2}.
For this, consider the homomorphism
\begin{gather*}
\rho_C\colon \  \operatorname{Aut}(X) \longrightarrow  \operatorname{Aut}({\mathcal M}_C)
\end{gather*}
def\/ined by sending $h \in  \operatorname{Aut}(X)$ to the automorphism
of ${\mathcal M}_C$ given by
$(L ,D)  \longmapsto (h^*L ,h^*D)$. We show in Section~\ref{sec3-1} that $\rho_C$ is injective if $g \geq  2$.
Let $G$ denote the subgroup of $\operatorname{Aut}({\mathcal M}_C)$ generated by
$\rho_C(\operatorname{Aut}(X))$ together with the inversion $(L ,D) \longmapsto
(L^\vee ,D^\vee)$ of the group ${\mathcal M}_C$; we denote the dual of a vector
bundle, a vector space or a homomorphism by the superscript~``$\vee$''. Using the
actions of~$G$ and $\rho_C(\operatorname{Aut}(X))$ on ${\mathcal M}_C$, consider the
semi-direct products
\begin{gather*}
{\mathcal G}_0  :=  {\mathcal M}_C\rtimes\rho_C(\operatorname{Aut}(X))
\qquad \text{and} \qquad {\mathcal G}  :=  {\mathcal M}_C\rtimes G  .
\end{gather*}
Through these groups we can characterize $\operatorname{Aut}_\theta({\mathcal M}_C)$ (see Theorem~\ref{thm2}):

\begin{Theorem} The group $\operatorname{Aut}_\theta({\mathcal M}_C)$ is given by
\begin{enumerate}\itemsep=0pt
\item[$1)$] $\operatorname{Aut}_\theta({\mathcal M}_C)
 =  {\mathcal G}$ if $X$ is not hyperelliptic;

\item[$2)$] $\operatorname{Aut}_\theta({\mathcal M}_C)  =  {\mathcal G}_0$ if $X$ is hyperelliptic.
\end{enumerate}
\end{Theorem}

As a Corollary, we deduce that any automorphism of ${\mathcal M}_C$ preserving the
cohomology class $\theta$ actually preserves the symplectic form and so the above
theorem also gives the group of algebraic symplectomorphisms of~${\mathcal M}_C$.

In Section \ref{sec4} we consider the Betti moduli space ${\mathcal
M}_R$ of representations of $\pi_1(X)$
into the multiplicative group $\mathbb{C}^*$ (following~\cite{simpson}). The space ${\mathcal M}_R  =  \operatorname{Hom}( \pi_1(X) ,
\mathbb{C}^*)$, which is
isomorphic to $( \mathbb{C}^* )^{2g}$. The group $\Gamma$ of automorphisms of
the $\mathbb{Z}$-module $H_1(X,\mathbb{Z})$ is isomorphic to ${\rm GL}(2g ,
 \mathbb{Z})$, and thus there is a natural map
\begin{gather*}
f\colon \  \operatorname{Aut}({\mathcal M}_R) \longrightarrow  \Gamma,
\end{gather*}
that sends an automorphism of ${\mathcal M}_R$ to its induced
action on $H_1(X , \mathbb{Z})$. In Section~\ref{sec4}, we show that $f$ admits a
right-splitting so that $\operatorname{Aut}({\mathcal M}_R)= \text{kernel}(f) \rtimes \Gamma$. Moreover, since the kernel of~$f$ is given by the natural action of ${\mathcal M}_R =
 (\mathbb{C}^*)^{2g}$ on itself by translations, we obtain that (see Theorem~\ref{thm3}):

\begin{Theorem}
The automorphism group $\operatorname{Aut}({\mathcal M}_R)$ is the semi-direct product
${\mathcal M}_R\rtimes\Gamma$.
\end{Theorem}

As with the de Rham moduli space, non-abelian Hodge theory determines a natural
symplectic form on~${\mathcal M}_R$. We f\/ind that the subgroup of $\operatorname{Aut}({\mathcal
M}_R)$ preserving this form is given by ${\mathcal M}_R \rtimes \Gamma_{\rm Sp}$,
where $\Gamma_{\rm Sp}$ is the subgroup of $\Gamma$ preserving the cap product on
$H_1(X ,  \mathbb{Z})$, so $\Gamma_{\rm Sp}$ is isomorphic to the symplectic group
$\text{Sp}(2g ,  \mathbb{Z})$.

Finally, in  Section \ref{sec5} we study the Dolbeault moduli space ${\mathcal M}_H$ of degree zero Higgs line bundles, that is pairs $(L,\Phi)$, where $L$ is a degree zero line bundle on $X$ and~$\Phi$ is a holomorphic $1$-form on~$X$. This moduli space is the holomorphic cotangent bundle
$T^\vee J(X)$ of the Jacobian $J(X)$. Considering the isomorphism $T^\vee J(X)
 =  J(X) \times H^0(X ,  K_X)$, where $K_X$ is the holomorphic cotangent bundle of $X$ we obtain that (see Lemma~\ref{lH}):

\begin{Lemma}
Any $f \in  \operatorname{Aut}({\mathcal M}_H)$ is of the form
\begin{gather*}
f =  f_1\times f_2  ,
\end{gather*}
where $f_1 \in  \operatorname{Aut}(J(X))$ and $f_2 \in  \operatorname{Aut}(H^0(X, K_X))$.
\end{Lemma}

Since the moduli space ${\mathcal M}_H$ is the cotangent bundle of $J(X)$, it carries
a canonical symplectic form $\theta$. We shall denote by $\operatorname{Aut}_\theta({\mathcal M}_H)$ the
subgroup of $\operatorname{Aut}({\mathcal M}_H)$ preserving~$\theta$, and let~$\Omega_{J(X)}$
denote the holomorphic cotangent bundle of $J(X)$. Recalling that there is an isomorphism
$H^0( X ,  K_X )  =  H^0( J(X) , \Omega_{J(X)} )$, we conclude the paper showing that (see Theorem~\ref{thm4}):

\begin{Theorem}
The group $\operatorname{Aut}_\theta({\mathcal M}_H)$ is the semi-direct product
\begin{gather*}
H^0( J(X) ,
\Omega_{J(X)} ) \rtimes \operatorname{Aut}( J(X) ) ,
\end{gather*} where $\operatorname{Aut}( J(X) )$ acts on $H^0( J(X) ,
\Omega_{J(X)})$ by $f\cdot \alpha =  (f^{-1})^*(\alpha)$, for $f  \in \operatorname{Aut}(J(X))$,
$\alpha  \in  H^0(J(X) , \Omega_{J(X)} )$.
\end{Theorem}

\section[Structure of the moduli space of ${\mathbb C}^*$-connections]{Structure of the moduli space of $\boldsymbol{{\mathbb C}^*}$-connections}\label{sec2}

Let $X$ be a compact connected Riemann surface of genus $g  \geq  1$, and $K_X$
its holomorphic cotangent bundle.
The Jacobian of~$X$, which parametrizes all the isomorphism classes of
holomorphic line bundles on $X$ of degree zero, is denoted by~$J(X)$.
Let ${\mathcal M}_C$ be the moduli space of holomorphic connections on $X$ of
rank one. Therefore, ${\mathcal M}_C$ parametrizes the isomorphism classes of pairs
of the form $(L ,D)$, where $L$ is a holomorphic line bundle on $X$ and $D$ is
a holomorphic connection on $L$. Since there are no nonzero $(2 ,0)$-forms on
$X$, any holomorphic connection on~$X$ is automatically integrable.

The adjoint action of the algebraic group ${\mathbb C}^*$ on its Lie algebra
$\operatorname{Lie}({\mathbb C}^*) =  {\mathbb C}$ is trivial. Consequently, for
any $(L , D) \in  {\mathcal M}_C$, the holomorphic tangent bundle
to ${\mathcal M}_C$ at the point $(L , D)$ is
\begin{gather}\label{htbc}
T_{(L,D)} {\mathcal M}_C =  H^1(X, {\mathbb C}) .
\end{gather}
Therefore, the real tangent bundle $T^{\mathbb R}_{(L,D)} {\mathcal M}_C$ is identif\/ied
with $H^1(X, {\mathbb C})$, and the almost complex structure on
$T^{\mathbb R}_{(L,D)} {\mathcal M}_C =  H^1(X, {\mathbb C})$ is multiplication
by $\sqrt{-1}$.

Since any holomorphic connection on $X$ is f\/lat, the degree
of any holomorphic line bundle admitting a
holomorphic connection is zero. Therefore, we have an algebraic morphism
\begin{gather}\label{e1}
\varphi\colon \  {\mathcal M}_C \longrightarrow  J(X)  ,\qquad
(L  ,D) \longmapsto  L  .
\end{gather}
This map $\varphi$ is surjective because any holomorphic line bundle~$L$ on~$X$
of degree zero admits a~holomorphic connection. More precisely, the space of all
holomorphic connections on~$L$ is an af\/f\/ine space for the vector space
$H^0(X, K_X)$. Therefore, $\varphi$ makes ${\mathcal M}_C$ an
algebraic principal $H^0(X, K_X)$-bundle over $J(X)$.

Let ${\mathbb V}$ denote the trivial holomorphic vector bundle $J(X)\times H^0(X, K_X)$
over $J(X)$ with f\/iber $H^0(X, K_X)$.
The isomorphism classes of algebraic principal $H^0(X, K_X)$-bundles over $J(X)$
are parametrized by $H^1(J(X), {\mathbb V})$. We will calculate the cohomology
class corresponding to~${\mathcal M}_C$. Note that $\varphi$ does not admit any
holomorphic section because $J(X)$ is compact and ${\mathcal M}_C$ is biholomorphic
to $({\mathbb C}^*)^{2g}$ thus ruling out the existence of any nonconstant holomorphic
map from $J(X)$ to ${\mathcal M}_C$. Consequently, the class in
$H^1(J(X), {\mathbb V})$ corresponding to~${\mathcal M}_C$ is nonzero.

We will brief\/ly describe the Dolbeault type construction of cohomological invariants for
principal $H^0(X,  K_X)$-bundles.

Take an algebraic principal $H^0(X, K_X)$-bundle $q\colon E \longrightarrow
J(X)$. Choose a $C^\infty$ section
\begin{gather*}
s\colon \ J(X) \longrightarrow  E
\end{gather*}
for $q$; such a section exists because the f\/ibers of the projection $q$ are contractible. If
$s$ is holomorphic, then the holomorphic principal $H^0(X, K_X)$-bundle $E$ is trivial.
The invariant for~$E$ is a measure of the failure of~$s$ to be holomorphic. To explain this,
let~$J_1$ and~$J_2$ denote the almost complex structures on $J(X)$ and $E$ respectively.
Let $ds\colon  T^{\mathbb R}J(X) \longrightarrow T^{\mathbb R} E$ be the dif\/ferential of
the map $s$. For any $x \in  J(X)$ and $y \in  E_x$, consider the homomorphism
\begin{gather}\label{g1}
T^{\mathbb R}_xJ(X) \longrightarrow  T^{\mathbb R}_y E  , \qquad
v \longmapsto  ds(J_1(v))- J_2(ds(v)) .
\end{gather}
Since
\begin{itemize}\itemsep=0pt
\item $q\circ s =  \operatorname{Id}_{J(X)}$, and

\item the map $q$ is holomorphic,
\end{itemize}
it follows that the tangent vector $ ds(J_1(v))- J_2(ds(v))$ in~\eqref{g1} is vertical
for~$q$. Using the action of the group $H^0(X,  K_X)$ on $E$, the vertical tangent bundle
for~$q$ is the trivial vector bundle with f\/iber $H^0(X, K_X)$.
Consequently, the homomorphism in~\eqref{g1} def\/ines a section
\begin{gather*}
c(E,s) \in  C^\infty\big(J(X),  \Omega^{0,1}_{J(X)}\otimes {\mathbb V}\big)  .
\end{gather*}
This $(0 ,1)$-form $c(E,s)$ is $\overline{\partial}$-closed because $E$ is a~holomorphic principal $H^0(X, K_X)$-bundle. Then, the Dolbeault cohomological class
\begin{gather}\label{invariant}
c(E) \in  H^1(J(X), {\mathbb V})
\end{gather}
def\/ined by it is the invariant for~$E$.

The Lie algebra $\operatorname{Lie}(J(X))$ of $J(X)$ is the abelian algebra
$H^1(X,  {\mathcal O}_X)$. The Serre duality theorem says that
$H^1(X,  {\mathcal O}_X) =  H^0(X,  K_X)^\vee$. Therefore, the vector bundle
$\mathbb V$ is identif\/ied with the holomorphic cotangent bundle
$\Omega_{J(X)}$. Consequently, we have
\begin{gather*}
H^1(J(X),  {\mathbb V}) =  H^1(J(X),  \Omega_{J(X)})  .
\end{gather*}
Hence the isomorphism classes of holomorphic principal $H^0(X, K_X)$-bundles
on $J(X)$ are pa\-ra\-met\-ri\-zed by $H^1(J(X), \Omega_{J(X)})$. We note that every
element of $H^1(J(X), \Omega_{J(X)})$ is the inva\-riant~(\ref{invariant}) for some holomorphic principal
$H^0(X, K_X)$-bundles on $J(X)$.

Let ${\mathcal M}_R :=  \operatorname{Hom}(\pi_1(X, x_0), {\mathbb C}^*) =
\operatorname{Hom}(H_1(X,  {\mathbb Z}),  {\mathbb C}^*)$ be the space of $1$-dimensional
representations. Sending a f\/lat connection to its monodromy representation, we get a
holomorphic isomorphism
\begin{gather*}
f\colon \  {\mathcal M}_C  \stackrel{\sim}{\longrightarrow}  {\mathcal M}_R  .
\end{gather*}
We have $\operatorname{Hom}(H_1(X, {\mathbb Z}), {\rm U}(1)) \hookrightarrow
{\mathcal M}_R$ using the inclusion of ${\rm U}(1)  =  S^1$ in ${\mathbb C}^*$. From
Hodge theory it follows that every $L \in  J(X)$ admits a unique holomorphic
connection such that the monodromy lies in ${\rm U}(1)$, and thus the composition
\begin{gather}\label{e-1}
\operatorname{Hom}(H_1(X,  {\mathbb Z}),  {\rm U}(1)) \stackrel{f^{-1}}{\longrightarrow}
{\mathcal M}_C \stackrel{\varphi}{\longrightarrow}  J(X)
\end{gather}
is a dif\/feomorphism, where $\varphi$ is constructed in~\eqref{e1}. We note that
the above composition $\varphi\circ f^{-1}$ is a dif\/feomorphism because it is bijective
and homomorphism of groups. We shall denote by
\begin{gather}\label{e2}
\xi\colon \  J(X) \longrightarrow  {\mathcal M}_C
\end{gather}
 the $C^\infty$ section of $\varphi$ given by the inverse of the composition
in~\eqref{e-1}.

Given any $L \in  J(X)$, we consider $\nabla =  \nabla^{1,0}+ \nabla^{0,1}$
the unique unitary f\/lat connection on $L$ such that $(0 ,1)$-type component
$\nabla^{0,1}$ is the Dolbeault operator
on~$L$. The real tangent space $T^{\mathbb R}_{\xi(L)} {\mathcal M}_C$ is
$H^1(X,  {\mathbb C})$, and the almost complex structure on $T^{\mathbb R}_{\xi(L)}
{\mathcal M}_C$ coincides with the multiplication by $\sqrt{-1}$ on
$H^1(X,  {\mathbb C})$ (see \eqref{htbc} and the sentence following it).
Therefore, the holomorphic tangent space
to ${\mathcal M}_C$ is identif\/ied with $H^1(X, {\mathbb C})$. The
inclusion of the Lie group ${\rm U}(1) \hookrightarrow {\mathbb C}^*$, identif\/ies
the Lie algebra $\operatorname{Lie}({\rm U}(1))$ with the subspace
\begin{gather*}
\sqrt{-1}\mathbb R \subset \operatorname{Lie}({\mathbb C}^*) =  \mathbb C .
\end{gather*}
Therefore, the subspace
\begin{gather*}
T^{\mathbb R}_{\xi(L)} \operatorname{Hom}(H_1(X,  {\mathbb Z}),  {\rm U}(1))
  \subset  T^{\mathbb R}_{\xi(L)} {\mathcal M}_C =  H^1(X,  {\mathbb C})
\end{gather*}
coincides with $H^1(X,  \sqrt{-1}{\mathbb R})$ equipped with its natural inclusion
\begin{gather*}H^1(X,  \sqrt{-1}{\mathbb R}) \hookrightarrow  H^1(X,  {\mathbb C})  .\end{gather*} The
anti-holomorphic tangent space $T^{0,1}_L
J(X)$ is identif\/ied with $H^0(X,  K_X)$ by sending any $\alpha \in
H^0(X,  K_X)$ to the f\/lat unitary connection
\begin{gather*}
\big(\nabla^{1,0}-\alpha\big) + \big(\nabla^{0,1}+\overline{\alpha}\big)  .
\end{gather*}
From the above, the complex structure on $T^{0,1}_L J(X)$ coincides with multiplication by
$\sqrt{-1}$ on $H^0(X,  K_X)$. If we identify $T^{0,1}_L J(X)$ with
 $T^{\mathbb R}_L J(X)$ by sending any $(0  ,1)$-tangent vector to its real part,
then the isomorphism
\begin{gather*}
T^{\mathbb R}_L J(X) \longrightarrow
T^{\mathbb R}_{\xi(L)} \operatorname{Hom}(H_1(X,  {\mathbb Z}),  {\rm U}(1))
\end{gather*}
given by the dif\/ferential of the composition map in \eqref{e-1} sends any
$\alpha \in  H^0(X,  K_X)$ to the element in
\begin{gather*}
-2\sqrt{-1}\cdot \text{Im}(\alpha) \in  H^1\big(X,  \sqrt{-1}{\mathbb R}\big)  .
\end{gather*}

The cup product
\begin{gather*}
\bigwedge\nolimits^2 H^1\big(X,  \sqrt{-1}{\mathbb R}\big) \longrightarrow
H^2(X,  {\mathbb R}) =  \mathbb R
\end{gather*}
produces a $2$-form $\omega$ on $J(X)$. The form $\omega$ is closed because
the translation action of~$J(X)$ on itself preserves $\omega$, and any translation
invariant form on a torus is closed. In fact, $\omega$ is a K\"ahler form on~$J(X)$. We shall let
\begin{gather}\label{e3}
\widehat{\omega} \in H^1(J(X),  \Omega_{J(X)})
\end{gather}
be the Dolbeault cohomology class represented by $\omega$.

 From the above, the anti-holomorphic tangent space $T^{0,1}_L
J(X)$ is identif\/ied with
\begin{gather*}
T^{\mathbb R}_{\xi(L)} \operatorname{Hom}(H_1(X,  {\mathbb Z}),  {\rm U}(1)),
\end{gather*}
a subspace of $H^1(X,  {\mathbb C})$ which in turn gives the holomorphic tangent space
to ${\mathcal M}_C$. Hence, consider the almost complex structures obtained for~$J(X)$ and~${\mathcal M}_C$, combined with the above description
of the dif\/ferential of~$\xi$, one has that the
class in $H^1(J(X),  {\mathbb V})$ corresponding to the principal
$H^0(X,  K_X)$-bundle ${\mathcal M}_C$ coincides with~$\widehat{\omega}$
in~\eqref{e3}.

Let
\begin{gather}\label{e4}
0 \longrightarrow  \Omega_{J(X)}  \stackrel{\iota}{\longrightarrow}  E
  \stackrel{\sigma}{\longrightarrow}  {\mathcal O}_{J(X)} \longrightarrow  0
\end{gather}
be the extension of ${\mathcal O}_{J(X)}$ by $\Omega_{J(X)}$ associated to the
extension class $\widehat{\omega}$ in~\eqref{e3}. The section of~${\mathcal O}_{J(X)}$
given by the constant function $1$ will be denoted by~$1_{J(X)}$. We note that
for the projection~$\sigma$ in~\eqref{e4}, the inverse image
$\sigma^{-1}(1_{J(X)}(J(X)))  \subset  E$ is a principal
$H^0(X,  K_X)$-bundle on $J(X)$ (recall that the dual vector space
$\operatorname{Lie}(J(X))^\vee$ is identif\/ied with
$H^0(X,  K_X)$). Since the class in $H^1(J(X),  \Omega_{J(X)})$
corresponding to the principal $H^0(X,  K_X)$-bundle ${\mathcal M}_C$
coincides with~$\widehat{\omega}$, we have the following:

\begin{Lemma}\label{lem1}
The variety ${\mathcal M}_C$ is algebraically isomorphic to the inverse image
\begin{gather*}
\sigma^{-1}(1_{J(X)}(J(X))).
\end{gather*}
\end{Lemma}

Through the above lemma, we can recover the following result, which from a dif\/ferent perspective can be deduced since the universal vector extension of the Jacobian parametrizes line bundles with connections \cite[Chapter~1]{MM}, and the universal vector extension of any abelian variety is anti-af\/f\/ine \cite[Proposition~2.3(i)]{Br}.

\begin{Proposition}\label{prop1}
There are no nonconstant algebraic functions on ${\mathcal M}_C$.
\end{Proposition}

\begin{proof}
In view of Lemma~\ref{lem1} it suf\/f\/ices to show that the variety
$\sigma^{-1}(1_{J(X)}(J(X)))$ does not admit any nonconstant algebraic function. We will
f\/irst express $\sigma^{-1}(1_{J(X)}(J(X)))$ as a hyper-plane complement $\mathcal{Y}$ in a
projective bundle over $J(X)$ in  Step~1. Then in Step~2 we shall study
associated bundles, which in turn allow us to study $H^0({\mathcal Y},  {\mathcal
O}_{\mathcal Y})$ in  Step~3. From the description of the cohomology group that we
obtain, we see that $\mathcal{Y}$ does not admit any nonconstant algebraic function if and
only if certain natural inclusion is surjective. Hence, in Step~4 we study this
inclusion, by taking the dual exact sequence to~\eqref{e4}. Surjectivity of the inclusion
can be then seen equivalent to injectivity of an associated map~$\beta$. We conclude the
proof of the proposition by showing in Step~5 that this map is indeed injective.

{\it Step 1.} Let $P(E) \longrightarrow  J(X)$ and $P(\Omega_{J(X)}) \longrightarrow  J(X)$
be the projective bundles parametrizing the lines in the f\/ibers of~$E$
(constructed in~\eqref{e4}) and $\Omega_{J(X)}$
respectively. The homomorphism~$\iota$ in~\eqref{e4} produces an embedding
\begin{gather*}
\widehat{\iota} \colon \   P(\Omega_{J(X)}) \longrightarrow P(E)  .
\end{gather*}
The divisor $\widehat{\iota}(P(\Omega_{J(X)}))  \subset  P(E)$ will be denoted
by $\mathcal D$. We have
\begin{gather}\label{e5}
{\mathcal Y}  :=  P(E){\setminus} {\mathcal D} \cong \sigma^{-1}(1_{J(X)}(J(X)))
\end{gather}
by sending any $v \in  \sigma^{-1}(1_{J(X)}(z))$ and $z \in  J(X)$, to the line
in the f\/iber $E_z$ generated by~$v$.

{\it Step 2.} Consider now the natural projection
\begin{gather*}
p \colon \  P(E) \longrightarrow J(X).
\end{gather*}
For ${\mathcal L} \longrightarrow  P(E)$ the dual of the
tautological line bundle, the f\/iber of~${\mathcal L}$ over any $y \in  P(E)$ is
the dual of the line in~$E_{p(y)}$ represented by~$y$.

Note that for any point $z \in  J(X)$, the two line bundles ${\mathcal L}\vert_{p^{-1}(z)}$
and ${\mathcal O}_{P(E)}({\mathcal D})\vert_{p^{-1}(z)}$ on $p^{-1}(z)$ are isomorphic.
Therefore, from the seesaw theorem (see \cite[p.~51, Corollary~6]{Mu}) it follows that
there is a holomorphic line bundle~$L_0$ on~$J(X)$ such that
\begin{gather}\label{g2}
{\mathcal O}_{P(E)}({\mathcal D}) =  {\mathcal L}\otimes p^*L_0  .
\end{gather}
By the adjunction formula \cite[p.~146]{GH}, the restriction of ${\mathcal O}_{P(E)}(
{\mathcal D})$ to $\mathcal D$ is the normal bund\-le~$N_{\mathcal D}$ to the divisor
${\mathcal D} \subset  P(E)$. This normal bundle~$N_{\mathcal D}$ is identif\/ied with
\begin{gather*}
\operatorname{Hom}\big(({\mathcal L}\vert_{\mathcal D})^\vee,  p^*_1(E/\Omega_{J(X)})\big) =
p^*_1(E/\Omega_{J(X)})\otimes ({\mathcal L}\vert_{\mathcal D}),
\end{gather*}
where
\begin{gather*}
p_1 =  p\vert_{\mathcal D}   \colon \  {\mathcal D} \longrightarrow  J(X)
\end{gather*}
is the restriction of~$p$.
Now, since the quotient $E/\Omega_{J(X)}$ is the trivial line
bundle (see~\eqref{e4}), it follows that $N_{\mathcal D}$ is isomorphic to ${\mathcal
L}\vert_{\mathcal D}$. Consequently from~\eqref{g2} it follows that the line bund\-le~$L_0$ is trivial. This in turn implies that
\begin{gather}\label{e6}
{\mathcal O}_{P(E)}({\mathcal D}) =  {\mathcal L}  .
\end{gather}

{\it Step 3}. To calculate $H^0({\mathcal Y},  {\mathcal O}_{\mathcal Y})$, note that
\begin{gather}\label{e7}
H^0({\mathcal Y},  {\mathcal O}_{\mathcal Y}) =  \varinjlim_{i\geq 0}
H^0(P(E),  {\mathcal O}_{P(E)}(i{\mathcal D})) =
\varinjlim_{i\geq 0} H^0\big(P(E),  {\mathcal L}^i\big)
\end{gather}
(see \eqref{e5} and \eqref{e6}). Since $\mathcal D$ is an ef\/fective divisor, from~\eqref{e5} and~\eqref{e7} we conclude that $\sigma^{-1}(1_{J(X)}(J(X)))$ does not
admit any nonconstant algebraic function if and only if the natural inclusion
\begin{gather}
H^0\big(P(E),  {\mathcal L}^i\big) =
H^0(P(E),  {\mathcal O}_{P(E)}(i{\mathcal D}))\nonumber
\\
\label{e8}
\hphantom{H^0\big(P(E),  {\mathcal L}^i\big)}{} \hookrightarrow
H^0(P(E),  {\mathcal O}_{P(E)}((i+1){\mathcal D})) =
H^0\big(P(E),  {\mathcal L}^{i+1}\big)
\end{gather}
is surjective for all $i \geq  0$. Note that
\begin{gather*}
H^0(P(E),  {\mathcal O}_{P(E)}(i{\mathcal D})) =  H^0\big(J(X),  \operatorname{Sym}^i\big(E^\vee\big)\big)  .
\end{gather*}

{\it Step 4.} To prove that the homomorphism in \eqref{e8} is indeed surjective, consider the dual
of the exact sequence in \eqref{e4}:
\begin{gather*}
0 \longrightarrow  {\mathcal O}_{J(X)}  \stackrel{\sigma^\vee}{\longrightarrow}  E^\vee
 \stackrel{\iota^\vee}{\longrightarrow}  TJ(X)  \longrightarrow  0  .
\end{gather*}
Taking its $(i+1)$-th symmetric power, we have
\begin{gather}\label{e9}
0 \longrightarrow  \operatorname{Sym}^i\big(E^\vee\big)  \stackrel{\sigma'}{\longrightarrow}
\operatorname{Sym}^{i+1}\big(E^\vee\big)  \stackrel{\operatorname{Sym}^{i+1}(\iota^\vee)}{\longrightarrow}
 \operatorname{Sym}^{i+1}(TJ(X)) \longrightarrow  0  ,
\end{gather}
where $\operatorname{Sym}^{i+1}(\iota^\vee)$ is the homomorphism of symmetric products induced
by the homomor\-phism~$\iota^\vee$; the above homomorphism~$\sigma'$ is the symmetrization of
the homomorphism
\begin{gather*}
\otimes^{i}E^\vee =  {\mathcal O}_{J(X)}\otimes \big(\otimes^{i}E^\vee\big)
 \stackrel{\sigma^\vee\otimes{\rm Id}}{\longrightarrow}  \otimes^{i+1}E^\vee  .
\end{gather*}
Let
\begin{gather*}
\beta \colon \  H^0\big(J(X),   \operatorname{Sym}^{i+1}(TJ(X))\big) \longrightarrow  H^1\big(J(X),
\operatorname{Sym}^i\big(E^\vee\big)\big)
\end{gather*}
be the connecting homomorphism in the long exact sequence of cohomologies
associated to the short exact sequence in~\eqref{e9}. Consider the homomorphisms
\begin{gather}
H^0\big(J(X),   \operatorname{Sym}^{i+1}(TJ(X))\big)  \stackrel{\beta}{\longrightarrow}  H^1\big(J(X),
\operatorname{Sym}^i\big(E^\vee\big)\big)\nonumber\\
\hphantom{H^0\big(J(X),   \operatorname{Sym}^{i+1}(TJ(X))\big) }{}
 \stackrel{\gamma}{\longrightarrow} H^1\big(J(X), \operatorname{Sym}^{i}(TJ(X))\big)  ,\label{e10}
\end{gather}
where $\gamma$ is induced by the homomorphism $\operatorname{Sym}^{i}(\iota^\vee)$ (see~\eqref{e9}).

From the long exact sequence of cohomologies for~\eqref{e9} it follows immediately that
the homomorphism in~\eqref{e8} is surjective if $\beta$ in~\eqref{e10} is injective. To
prove that~$\beta$ is injective, it is enough to show that the composition
$\gamma\circ\beta$ in~\eqref{e10} is injective.

{\it Step 5}. Since the extension class for~\eqref{e4} is the cohomology class
$\widehat{\omega}$, the extension class for~\eqref{e9} is $-(i+1)\widehat{\omega}$.
Consequently, the homomorphism $\gamma\circ\beta$ sends any
\begin{gather*}
\eta \in
H^0\big(J(X),   \operatorname{Sym}^{i+1}(TJ(X))\big)
\end{gather*} to the Dolbeault cohomology class of the
contraction
\begin{gather*}
\omega\otimes'\eta \in  C^\infty\big(J(X),  \Omega^{0,1}_{J(X)}\otimes
\operatorname{Sym}^{i+1}(TJ(X))\big)
\end{gather*}
of $\omega\otimes\eta$ of $\Omega^{1,0}_{J(X)}$ and $T(X)$; note that the tensor product
$\omega\otimes\eta$ is a~section of $\Omega^{0,1}_{J(X)}\otimes \Omega^{1,0}_{J(X)}
\otimes \operatorname{Sym}^{i+1}(TJ(X))$ and hence its contraction $\omega\otimes'\eta$ is a
section of $\Omega^{0,1}_{J(X)}\otimes \operatorname{Sym}^{i}(TJ(X)))$. Since both $\omega$ and
$\eta$ are invariant under translations of $J(X)$, it follows that $\omega\otimes'\eta$ is
also invariant under translations of $J(X)$, and hence
represents a nonzero cohomology class. The section $\omega\otimes'\eta$ is nonzero because
$\omega$ is pointwise nondegenerate (recall that it is a K\"ahler form). Therefore, we
conclude that the homomorphism $\gamma\circ\beta$ is injective. Hence the homomorphism in
\eqref{e8} is surjective, and the proof is complete.
\end{proof}

\section[Automorphisms of the moduli of ${\mathbb C}^*$-connections]{Automorphisms of the moduli of $\boldsymbol{{\mathbb C}^*}$-connections}\label{sec3}

The group of algebraic automorphisms of the variety ${\mathcal M}_C$ will be denoted by
$\operatorname{Aut}({\mathcal M}_C)$. The moduli space ${\mathcal M}_C$ is an algebraic group, with
group operation
\begin{gather*}
(L_1  ,D_1)\cdot (L_2  ,D_2) =  \big(L_1\otimes L_2  ,D_1\otimes \operatorname{Id}_{L_2}+
\operatorname{Id}_{L_1}\otimes D_2\big)  .
\end{gather*}
The algebraic map $\varphi$ in \eqref{e1} is a homomorphism of algebraic groups.

The translation action of ${\mathcal M}_C$ on itself
produces an injective homomorphism
\begin{gather*}
\rho  \colon \   {\mathcal M}_C \longrightarrow  \operatorname{Aut}({\mathcal M}_C)  .
\end{gather*}

\begin{Theorem}\label{thm1}
The quotient $\operatorname{Aut}({\mathcal M}_C)/(\rho({\mathcal M}_C))$ is a countable group.
In particular, the image of $\rho$ is the connected component of
$\operatorname{Aut}({\mathcal M}_C)$ containing the identity element.
\end{Theorem}

\begin{proof}
We will show that any automorphism of ${\mathcal M}_C$ descends to $J(X)$. For that,
f\/irst note that there is no nonconstant algebraic map from $\mathbb C$ to an abelian variety. Indeed, such a map would extend to a nonconstant algebraic map from ${\mathbb C}{\mathbb P}^1$,
and therefore some holomorphic $1$-form on the abelian variety would pull back to a
nonzero holomorphic $1$-form on ${\mathbb C}{\mathbb P}^1$, but ${\mathbb C}{\mathbb P}^1$
does not have any nonzero holomorphic $1$-form. Since there is no nonconstant algebraic
map from~$\mathbb C$ to~$J(X)$, there is no nonconstant algebraic map from a f\/iber of
$\varphi$ (see~\eqref{e1}) to the variety~$J(X)$, because the f\/ibers of
$\varphi$ are isomorphic to~${\mathbb C}^g$. This immediately implies
that any automorphism of~${\mathcal M}_C$ descends to an automorphism of~$J(X)$.

The group of all algebraic automorphisms of $J(X)$ will be denoted by
$\operatorname{Aut}(J(X))$. The above observation produces a homomorphism
\begin{gather}\label{delta}
\delta  \colon \ \operatorname{Aut}({\mathcal M}_C)  \longrightarrow \operatorname{Aut}(J(X))  .
\end{gather}
Recall that $\varphi$ in \eqref{e1} is a homomorphism of algebraic groups. Clearly,
we have \begin{gather*}
\rho(\text{kernel}(\varphi))  \subset  \text{kernel}(\delta)  .
\end{gather*}
We shall denote by \begin{gather*}
\rho_0  \colon \   J(X) \longrightarrow  \operatorname{Aut}(J(X))
\end{gather*}
the homomorphism given by the translation action of $J(X)$ on itself.

To prove the theorem, by the snake lemma it suf\/f\/ices to show the following two statements:
\begin{enumerate}\itemsep=0pt
\item[1)] the quotient $\operatorname{Aut}(J(X))/(\rho_0(J(X)))$ is a countable group,

\item[2)] the inclusion $\rho(\text{kernel}(\varphi))  \hookrightarrow
\text{kernel}(\delta)$ is surjective.
\end{enumerate}

The f\/irst statement follows from the fact that
$H^0(J(X),  TJ(X)) =  \operatorname{Lie}(J(X))$. In what follows we will prove the
second statement.

Take any $\psi \in  \text{kernel}(\delta)  \subset  \operatorname{Aut}({\mathcal M}_C)$, and
for $\zeta \in  H^0(X,  K_X)^\vee$, def\/ine the function
\begin{gather*}
F_{\psi,\zeta}  \colon \   {\mathcal M}_C \longrightarrow  {\mathbb C}  ,
\qquad z  \longmapsto  \zeta(\psi(z)-z)  .
\end{gather*}
Note that $\varphi(\psi(z)) =  \varphi(z)$ because $\psi \in
\text{kernel}(\delta)$, and hence
$\psi(z)-z \in  H^0(X,  K_X)$. From Proposition~\ref{prop1} we know
that $F_{\psi,\zeta}$ is a constant function. This implies that there is
an element $v \in  H^0(X,  K_X)$ such that
$\psi(z) =  z+ v$ for all $z \in  {\mathcal M}_C$. So we have $\psi \in
\rho(\text{kernel}(\varphi))$, which completes the proof.
\end{proof}

\subsection{Automorphisms preserving cohomology class}\label{sec3-1}

As mentioned previously, the moduli space ${\mathcal M}_C$ is equipped with an algebraic
symplectic form (see \cite{atiyah,Go}). The cohomology class in $H^2({\mathcal M}_C,
{\mathbb C})$ def\/ined by the symplectic form will be denoted by~$\theta$. The pullback
of the symplectic form on ${\mathcal M}_C$ by the section~$\xi$ in~\eqref{e2} coincides
with the K\"ahler form on~$J(X)$. Therefore, the cohomology class
$\theta$ on ${\mathcal M}_C$ coincides with the pull\-back~$\varphi^*\widehat{\omega}$
of the K\"ahler class on~$J(X)$ (see~\eqref{e1} and~\eqref{e3}). Let
$\operatorname{Aut}_\theta({\mathcal M}_C)$ denote the group of all $\tau \in
\operatorname{Aut}({\mathcal M}_C)$ such that $\tau^*\theta =  \theta$. Our aim in this
subsection is to compute $\operatorname{Aut}_\theta({\mathcal M}_C)$.

The group of all holomorphic automorphisms of $X$ will be denoted by $\operatorname{Aut}(X)$.
Let
\begin{gather*}
\operatorname{Aut}^0(X)  \subset  \operatorname{Aut}(X)
\end{gather*} be the connected component containing
the identity element. If $g \geq  2$, then we have \mbox{$\operatorname{Aut}^0(X) {=} e$}. Let
\begin{gather*}
\rho_C\colon \ \operatorname{Aut}(X) \longrightarrow  \operatorname{Aut}({\mathcal M}_C)
\end{gather*}
be the homomorphism that sends any $h  \in  \operatorname{Aut}(X)$ to the automorphism
of ${\mathcal M}_C$ def\/ined by
$(L ,D)  \longmapsto  (h^*L  ,h^*D)$. If $g  \geq  2$, then $\rho_C$ is injective.
Indeed, the homomorphism $\operatorname{Aut}(X) \longrightarrow  \operatorname{Aut}(J(X))$ that sends
any $h \in  \operatorname{Aut}(X)$ to the automorphism $L  \longmapsto  h^*L$ is injective
if $g \geq  2$. If $g = 1$ then $X$ is an elliptic curve and $\operatorname{Aut}^0(X) = X$, acting on itself by translations. If $\tau \colon X \to X$ is any such translation then for any line bundle with holomorphic connection~$(L,D)$, we have $(\tau^*L , \tau^*D)
 \cong  (L  ,D)$ since the corresponding f\/lat connections have the same monodromy. Therefore the homomorphism $\rho_C\vert_{\operatorname{Aut}^0(X)}$
is trivial, and $\rho_C$ produces an embedding of $\operatorname{Aut}(X)/\operatorname{Aut}^0(X)$
in $\operatorname{Aut}(J(X))$.

Let $G$ denote the subgroup of $\operatorname{Aut}({\mathcal M}_C)$ generated by
$\rho_C(\operatorname{Aut}(X))$ together with the inversion $(L  ,D)  \longmapsto
(L^\vee  ,D^\vee)$ of the group ${\mathcal M}_C$. Using the actions of $G$ and
$\rho_C(\operatorname{Aut}(X))$ on ${\mathcal M}_C$, we construct the semi-direct products
\begin{gather*}
{\mathcal G}_0  :=  {\mathcal M}_C\rtimes\rho_C(\operatorname{Aut}(X))
\qquad \text{and}\qquad {\mathcal G}  :=  {\mathcal M}_C\rtimes G  .
\end{gather*}
Note that using the action of $\rho_C(\operatorname{Aut}(X))$
(respectively, $G$) and the translation action of ${\mathcal M}_C$ on itself, the
group ${\mathcal G}_0$ (respectively, ${\mathcal G}$) acts on ${\mathcal M}_C$.

\begin{Theorem}\label{thm2}
The group $\operatorname{Aut}_\theta({\mathcal M}_C)$ is given by
\begin{enumerate}\itemsep=0pt
\item[$1)$] $\operatorname{Aut}_\theta({\mathcal M}_C)
 =  {\mathcal G}$ if $X$ is not hyperelliptic;

\item[$2)$] $\operatorname{Aut}_\theta({\mathcal M}_C)  =  {\mathcal G}_0$ if $X$ is hyperelliptic.
\end{enumerate}
\end{Theorem}

\begin{proof}
As mentioned before, we have $\theta =  \varphi^*\widehat{\omega}$. From this it
follows that for any element of ${\mathcal G}$, the corresponding automorphism of
${\mathcal M}_C$ preserves $\theta$.

Let $\operatorname{Aut}_{\widehat{\omega}}(J(X))$ be the group of all automorphisms of the variety
$J(X)$ that preserve the cohomology class $\widehat{\omega}$. From \cite[Hauptsatz, p.~35]{We}
one has the following:
\begin{enumerate}\itemsep=0pt
\item Assume that $X$ is not hyperelliptic. Then $\operatorname{Aut}_{\widehat{\omega}}(J(X))$
is generated by translations of $J(X)$, $\operatorname{Aut}(X)$ and the inversion
$L  \longmapsto  L^\vee$ of $J(X)$.

\item Assume that $X$ is hyperelliptic. Then $\operatorname{Aut}_{\widehat{\omega}}(J(X))$
is generated by translations of $J(X)$ and $\operatorname{Aut}(X)$. (The hyperelliptic
involution of $X$ induces the inversion of $J(X)$.)
\end{enumerate}

Consider the homomorphism $\delta$ in \eqref{delta}. In the proof of Theorem~\ref{thm1} it was shown that the inclusion
\begin{gather}\label{in}
\rho(\text{kernel}(\varphi))  \hookrightarrow  \text{kernel}(\delta)
\end{gather}
is surjective. First assume that $X$ is not hyperelliptic. Using $\varphi$ in
\eqref{e1}, we get a homomorphism
\begin{gather*}
G  \longrightarrow  \operatorname{Aut}(J(X))  .
\end{gather*}
{}From the above result of \cite{We} we know that this homomorphism is injective,
its image is a normal subgroup of $\operatorname{Aut}(J(X))$ and the composition
\begin{gather*}
G  \longrightarrow  \operatorname{Aut}(J(X)) \longrightarrow  \operatorname{Aut}(J(X))/J(X)
\end{gather*}
is an isomorphism. Therefore, from the surjectivity of the homomorphism in~\eqref{in}
we conclude that $\operatorname{Aut}_\theta({\mathcal M}_C)
 =  {\mathcal G}$.

If $X$ is hyperelliptic, then $\operatorname{Aut}(J(X))/\operatorname{Aut}(X) =  J(X)$ by the
above theorem of \cite{We}. Therefore, by the above argument it follows that
$\operatorname{Aut}_\theta({\mathcal M}_C)  =  {\mathcal G}_0$.
\end{proof}

From the def\/initions of $\mathcal{G}_0$ and $\mathcal{G}$, it is straightforward to verify that these groups preserve the algebraic symplectic
form on ${\mathcal M}_C$. Therefore, Theorem \ref{thm2} gives the following:

\begin{Corollary}\label{cor1}\quad
\begin{enumerate}\itemsep=0pt
\item[$1.$] Assume that $X$ is not hyperelliptic. Then the group of algebraic automorphisms of
${\mathcal M}_C$ that preserve the symplectic form on ${\mathcal M}_C$ is~${\mathcal G}$.

\item[$2.$] Assume that $X$ is hyperelliptic. Then the group of algebraic automorphisms of
${\mathcal M}_C$ that preserve the symplectic form on~${\mathcal M}_C$ is ${\mathcal G}_0$.
\end{enumerate}
\end{Corollary}

\section{Automorphisms of the representation space}\label{sec4}

The representation space
\begin{gather*}
{\mathcal M}_R =  \operatorname{Hom}(H_1(X,  {\mathbb Z}),  {\mathbb C}^*) =
\operatorname{Hom}(\pi_1(X),  {\mathbb C}^*)
\end{gather*}
is algebraically isomorphic to $({\mathbb C}^*)^{2g}$. A choice of a basis of
the $\mathbb Z$-module $H_1(X,  {\mathbb Z})$ produces an isomorphism of
${\mathcal M}_R$ with $({\mathbb C}^*)^{2g}$. The group structure of the multiplicative
group ${\mathbb C}^*$ makes~${\mathcal M}_R$ a complex algebraic group.

The group $H_1(X,  {\mathbb Z})$
is identif\/ied with $H^1({\mathcal M}_R,  {\mathbb Z})$ by the $(1,1)$-type
K\"unneth component of the f\/irst Chern class of a Poincar\'e line bundle on
$X\times{\mathcal M}_R$. It should be clarif\/ied that this $(1,1)$-type K\"unneth
component is independent of the choice of the Poincar\'e line bundle.
The group of all automorphisms of the $\mathbb Z$-module $H_1(X,  {\mathbb Z})$
will be denoted by~$\Gamma$. So $\Gamma$ is isomorphic to $\text{GL}(2g,{\mathbb Z})$.

Let $\operatorname{Aut}({\mathcal M}_R)$ denote the group of all algebraic automorphisms of
${\mathcal M}_R$. Let
\begin{gather}\label{f}
f  \colon \  \operatorname{Aut}({\mathcal M}_R) \longrightarrow  \Gamma
\end{gather}
be the homomorphism that sends any automorphism of ${\mathcal M}_R$ to the automorphism
of \begin{gather*}H^1({\mathcal M}_R,  {\mathbb Z}) =  H_1(X,  {\mathbb Z})\end{gather*} induced by it.

\begin{Lemma}\label{lem2}
The homomorphism $f$ in \eqref{f} is surjective.
\end{Lemma}

\begin{proof}
Given any $(a_{ij})^{2g}_{i,j=1} \in  \text{GL}(2g,{\mathbb Z})$, consider the
automorphism $T$ of $({\mathbb C}^*)^{2g}$ def\/ined as follows: the $i$-th coordinate
of $T(z_1,\dots  , z_{2g})$, $(z_1,\dots  , z_{2g}) \in  ({\mathbb C}^*)^{2g}$, is
\begin{gather*}
\prod_{j=1}^{2g} z^{a_{ij}}_j  .
\end{gather*}
The automorphism of $H^1(({\mathbb C}^*)^{2g}, {\mathbb Z}) = {\mathbb Z}^{2g}$ induced
by $T$ is given by the standard action of $(a_{ij})^{2g}_{i,j=1}$ on ${\mathbb Z}^{2g}$.
\end{proof}

The map $(a_{ij})^{2g}_{i,j=1}  \longmapsto  \operatorname{Aut}(({\mathbb C}^*)^{2g})$ in the
proof of Lemma~\ref{lem2} produces a canonical right-splitting of the homomorphism~$f$ in~\eqref{f}. Since $f$ is surjective, this implies that the group $\operatorname{Aut}({\mathcal
M}_R)$ is the semi-direct product
\begin{gather*}
\operatorname{Aut}({\mathcal M}_R) =  \text{kernel}(f)\rtimes\Gamma  .
\end{gather*}

The group of all algebraic automorphisms of $({\mathbb C}^*)^{2g}$ that preserve every
factor is $({\mathbb C}^*)^{2g}$ acting on itself by translations. Therefore,
we have the following:

\begin{Theorem}\label{thm3}
The automorphism group $\operatorname{Aut}({\mathcal M}_R)$ is the semi-direct product
${\mathcal M}_R\rtimes\Gamma$.
\end{Theorem}

We shall consider now $\Gamma_{\rm Sp}  \subset  \Gamma =
\operatorname{Aut}(H_1(X,  {\mathbb Z}))$ the group of automorphisms that preserve the
cap product on $H_1(X,  {\mathbb Z})$. So $\Gamma_{\rm Sp}$ is isomorphic to the
symplectic group $\text{Sp}(2g,{\mathbb Z})$. Using Theorem \ref{thm3} it can be
deduced that the group of automorphisms of ${\mathcal M}_R$ that preserve its
symplectic form is ${\mathcal M}_R\rtimes\Gamma_{\rm Sp}$.

Although the holomorphic isomorphism between ${\mathcal M}_C$ and
${\mathcal M}_R$ is not algebraic, comparing Theorem~\ref{thm2} and Theorem~\ref{thm3}
we obtain a relation between their automorphism groups. Let $h \colon \operatorname{Aut}(X)/\operatorname{Aut}^0(X) \times \mathbb{Z}_2 \to \operatorname{Aut}(\mathcal{M}_R)$ be the homomorphism which sends $\operatorname{Aut}(X)/\operatorname{Aut}^0(X)$ to its image in $\operatorname{Aut}(\mathcal{M}_R)$ and maps the generator of~$\mathbb{Z}_2$ to the inversion map $\phi   \longmapsto  \phi^{-1}$, sending a homomorphism $\phi \colon \pi_1(X) \to \mathbb{C}^*$ to its inverse~$\phi^{-1}$. Let $G_M \subset \operatorname{Aut}(\mathcal{M}_R)$ be the image of~$h$. Then:

\begin{Corollary}\label{cor2}
For any $\tau \in  \operatorname{Aut}_\theta ({\mathcal M}_C)$, the holomorphic
automorphism of ${\mathcal M}_R$ given by $\tau$ using the holomorphic identification
between ${\mathcal M}_C$ and ${\mathcal M}_R$ is actually algebraic. More precisely, the
set of automorphisms of ${\mathcal M}_R$ given by $\operatorname{Aut}_\theta ({\mathcal M}_C)$
is as follows:
\begin{enumerate}\itemsep=0pt
\item[$1)$] it is ${\mathcal M}_R\rtimes G_M$ if~$X$ is not hyperelliptic,

\item[$2)$] it is ${\mathcal M}_R\rtimes h(\operatorname{Aut}(X)/\operatorname{Aut}^0(X))$ if $X$ is hyperelliptic.
\end{enumerate}
\end{Corollary}

\begin{proof}
Recall from Theorem \ref{thm2} that $\operatorname{Aut}_\theta ({\mathcal M}_C)$ is generated by the translation action of~$\mathcal{M}_C$ on itself together with the action of $\operatorname{Aut}(X)$ and the inversion $(L  ,D)  \longmapsto  (L^\vee  ,D^\vee)$ of the group~${\mathcal M}_C$. In the case that $X$ is hyperelliptic the action of inversion coincides with the hyperelliptic involution, so may be omitted. As abstract groups, $\mathcal{M}_C$ and $\mathcal{M}_R$ are isomorphic, so the translation action of $\mathcal{M}_C$ coincides with the translation action of~$\mathcal{M}_R$, hence is algebraic with respect to~$\mathcal{M}_R$. It is also clear that the action of $\operatorname{Aut}(X)$ together with the inversion $(L  ,D)  \longmapsto  (L^\vee  ,D^\vee)$ act on $\mathcal{M}_R = \operatorname{Hom}(H_1(X,  {\mathbb Z}),  {\mathbb C}^*)$ as a subgroup of $\Gamma_{\rm Sp}$, hence are also algebraic with respect to~$\mathcal{M}_R$. This proves the claim that any $\tau \in  \operatorname{Aut}_\theta ({\mathcal M}_C)$ acts as an algebraic automorphism of~$\mathcal{M}_R$ and hence def\/ines natural homomorphism $j\colon \operatorname{Aut}_\theta ({\mathcal M}_C) \to \operatorname{Aut}(\mathcal{M}_C)$.

We claim that $j$ is injective. For this note that the restriction of $f$ to $G_M$ def\/ines a homomorphism $f \colon G_M \to \Gamma_{\rm Sp}$ which sends an automorphism of $X$ to its induced action on $H_1(X,\mathbb{Z})$ and sends the inversion map to $-\operatorname{Id}$. If $X$ is not hyperelliptic, then the composition $f \circ h \colon \operatorname{Aut}(X)/\operatorname{Aut}^0(X) \times \mathbb{Z}_2 \mapsto \Gamma_{\rm Sp}$ is injective and if $X$ is hyperelliptic then $f \circ h|_{\operatorname{Aut}(X)/\operatorname{Aut}^0(X)} \colon \operatorname{Aut}(X)/\operatorname{Aut}^0(X) \to \Gamma_{\rm Sp}$ is injective (for $g=1$ this is trivial, while for $g \ge 2$ this follows from, e.g.,~\cite[Section~V.2]{FK}). This proves the claim that $j$ is injective and that the image of~$j$ is ${\mathcal M}_R\rtimes G_M$ if $X$ is not hyperelliptic and ${\mathcal M}_R\rtimes h(\operatorname{Aut}(X)/\operatorname{Aut}^0(X))$ if~$X$ is hyperelliptic.
\end{proof}

\section{Automorphisms of moduli space of Higgs line bundles}\label{sec5}

The moduli space of Higgs line bundles on $X$ of degree zero is the Cartesian
product
\begin{gather*}
{\mathcal M}_H =  J(X)\times H^0(X,  K_X)  .
\end{gather*}
Let $\operatorname{Aut}({\mathcal M}_H)$ denote the group of all algebraic automorphisms
of the variety ${\mathcal M}_H$.

\begin{Lemma}\label{lH}
Any $f \in  \operatorname{Aut}({\mathcal M}_H)$ is of the form
\begin{gather*}
f =  f_1\times f_2  ,
\end{gather*}
where $f_1 \in  \operatorname{Aut}(J(X))$ and $f_2 \in  \operatorname{Aut}(H^0(X,  K_X))$.
\end{Lemma}

\begin{proof}
Let
\begin{gather*}\phi_1  \colon \  {\mathcal M}_H =  J(X)\times H^0(X,  K_X) \longrightarrow
J(X)  , \\
 \phi_2  \colon \  J(X)\times H^0(X,  K_X) \longrightarrow
H^0(X,  K_X)
\end{gather*}
be the natural projections. As noted before, there are no nonconstant
algebraic maps from $H^0(X,  K_X)$ to $J(X)$. So given $f$, there is a
unique automorphism
\begin{gather*}f_1 \in  \operatorname{Aut}(J(X))\end{gather*}
such that $f_1\circ\phi_1 =  \phi_1\circ f$.

Note that given $v \in  H^0(X,  K_X)$, one can def\/ine the map
\begin{align*}
\widehat{v} \colon \ J(X)  &\longrightarrow H^0(X, K_X), \\
 z &\longmapsto \phi_2(f(z,v)),
\end{align*}
 which is a constant map since it is holomorphic. Denoting by $v' \in  H^0(X,  K_X)$ the constant image of $\widehat{v}$, one can see that
\begin{align*}
f_2 \colon \  H^0(X, K_X)&\longrightarrow   H^0(X,  K_X),  \\
v&\longmapsto   v'
\end{align*}
is an automorphism, and thus one has that $f  =  f_1\times f_2$.
\end{proof}

The moduli space ${\mathcal M}_H$ can be naturally identif\/ied with the cotangent
bundle of~$J(X)$, hence it carries a canonical symplectic form~$\theta$. Let $\operatorname{Aut}_\theta({\mathcal M}_H)$ be the subgroup of $\operatorname{Aut}({\mathcal M}_H)$ preser\-ving~$\theta$. Recall that $\Omega_{J(X)}$ denotes the holomorphic cotangent bundle of
$J(X)$ and that there is a~naturally def\/ined isomorphism $H^0( X,  K_X ) =  H^0( J(X),
\Omega_{J(X)} )$.

\begin{Theorem}\label{thm4}
The group $\operatorname{Aut}_\theta({\mathcal M}_H)$ is the semi-direct product
\begin{gather*}
H^0( J(X) ,  \Omega_{J(X)} ) \rtimes \operatorname{Aut}( J(X) )  ,
\end{gather*} where
$\operatorname{Aut}( J(X) )$ acts on $H^0( J(X) , \Omega_{J(X)} )$ by $f \cdot \alpha  =
(f^{-1})^*(\alpha)$, for $f  \in  \operatorname{Aut}(J(X))$, $\alpha  \in  H^0(J(X) ,
  \Omega_{J(X)} )$.
\end{Theorem}

\begin{proof}
By Lemma~\ref{lH}, any automorphism of $T^\vee J(X)  =  J(X) \times H^0( J(X),
\Omega_{J(X)})$ has the form $f(x,y) =  ( f_1(x) , f_2(y) )$ for $f_1  \in
\operatorname{Aut}(J(X))$, $f_2  \in  \operatorname{Aut}( H^0(J(X) , \Omega_{J(X)})$. Since $f_1
 \in  \operatorname{Aut}(J(X))$, the derivative
$(f_1)_*(x) \colon H^0(J(X) ,  \Omega_{J(X)})^\vee  \longrightarrow
 H^0( J(X) ,  \Omega_{J(X)})^\vee$ is
independent of $x$ and will be denoted by $A$. Then, it is clear that $f_1 \times f_2$
preserves the symplectic form on $T^\vee J(X)$ if and only if $(f_2)_*(y)  =  (A^\vee)^{-1}$
for all $y$. Thus $f_2$ is an af\/f\/ine transformation of the form $f_2(y) =  (A^\vee)^{-1}y
+ y_0$, for some $y_0  \in  H^0( J(X) ,  \Omega_{J(X)})$. So $f_1 \times f_2$ is the
composition of $(f_1^*)^{-1}\colon T^\vee J(X) \longrightarrow  T^\vee J(X)$ with a translation
by $y_0$ in the f\/ibers of $T^\vee J(X) \longrightarrow  J(X)$. It follows easily that $\operatorname{Aut}_\theta({\mathcal
M}_H)$ is the semi-direct product $H^0( J(X) , \Omega_{J(X)}) \rtimes \operatorname{Aut}( J(X))$.
\end{proof}

\subsection*{Acknowledgements}

We thank the referees for helpful comments.
The f\/irst author is supported by a J.C.~Bose Fellowship. The third author
is partially supported by NSF grant DMS-1509693.

\pdfbookmark[1]{References}{ref}
\LastPageEnding


\begin{thebibliography}{999}
\footnotesize\itemsep=0pt

\bibitem{atiyah}
Atiyah M.F., Bott R., The {Y}ang--{M}ills equations over {R}iemann surfaces,
  \href{http://dx.doi.org/10.1098/rsta.1983.0017}{\textit{Philos. Trans. Roy. Soc. London Ser.~A}} \textbf{308} (1983),
  523--615.

\bibitem{Ba}
Baraglia D., Classif\/ication of the automorphism and isometry groups of {H}iggs
  bundle moduli spaces, \href{http://arxiv.org/abs/1411.2228}{arXiv:1411.2228}.

\bibitem{BS1}
Baraglia D., Schaposnik L.P., Higgs bundles and {$(A,B,A)$}-branes,
  \href{http://dx.doi.org/10.1007/s00220-014-2053-6}{\textit{Comm. Math. Phys.}} \textbf{331} (2014), 1271--1300,
  \href{http://arxiv.org/abs/1305.4638}{arXiv:1305.4638}.

\bibitem{BS2}
Baraglia D., Schaposnik L.P., Real structures on moduli spaces of Higgs
  bundles, \textit{Adv. Theor. Math. Phys.}, {t}o appear, \href{http://arxiv.org/abs/1309.1195}{arXiv:1309.1195}.

\bibitem{Br}
Brion M., Anti-af\/f\/ine algebraic groups, \href{http://dx.doi.org/10.1016/j.jalgebra.2008.09.034}{\textit{J.~Algebra}} \textbf{321}
  (2009), 934--952, \href{http://arxiv.org/abs/0710.5211}{arXiv:0710.5211}.

\bibitem{FK}
Farkas H.M., Kra I., Riemann surfaces, \href{http://dx.doi.org/10.1007/978-1-4612-2034-3}{\textit{Graduate Texts in Mathematics}},
  Vol.~71, Springer-Verlag, New York~-- Berlin, 1980.

\bibitem{Go}
Goldman W.M., The symplectic nature of fundamental groups of surfaces,
  \href{http://dx.doi.org/10.1016/0001-8708(84)90040-9}{\textit{Adv. Math.}} \textbf{54} (1984), 200--225.

\bibitem{GH}
Grif\/f\/iths P., Harris J., Principles of algebraic geometry, \href{http://dx.doi.org/10.1002/9781118032527}{\textit{Wiley Classics
  Library}}, John Wiley \& Sons, Inc., New York, 1994.

\bibitem{KW}
Kapustin A., Witten E., Electric-magnetic duality and the geometric {L}anglands
  program, \href{http://dx.doi.org/10.4310/CNTP.2007.v1.n1.a1}{\textit{Commun. Number Theory Phys.}} \textbf{1} (2007), 1--236,
  \href{http://arxiv.org/abs/hep-th/0604151}{hep-th/0604151}.

\bibitem{MM}
Mazur B., Messing W., Universal extensions and one dimensional crystalline
  cohomology, \href{http://dx.doi.org/10.1007/BFb0061628}{\textit{Lecture Notes in Math.}}, Vol.~370, Springer-Verlag,
  Berlin~-- New York, 1974.

\bibitem{Mu}
Mumford D., Abelian varieties, \textit{Tata Institute of Fundamental Research
  Studies in Mathematics}, Vol.~5, Hindustan Book Agency, New Delhi, 2008.

\bibitem{simpson}
Simpson C.T., Moduli of representations of the fundamental group of a smooth
  projective variety.~{II}, \href{http://dx.doi.org/10.1007/BF02698895}{\textit{Inst. Hautes \'Etudes Sci. Publ. Math.}}
  (1994), 5--79.

\bibitem{We}
Weil A., Zum {B}eweis des {T}orellischen {S}atzes, \textit{Nachr. Akad. Wiss.
  G\"ottingen Math.-Phys. Kl.~IIa} \textbf{1957} (1957), 33--53.

\end{thebibliography}
\end{document}